\documentclass{ws-procs9x6}            % default, citations in superscript
\usepackage{dsfont}
\usepackage{color}

\begin{document}
\title{Notes on the Chernoff estimate}

\author{Valentin A.Zagrebnov$^*$}

\address{Institut de Math\'{e}matiques de Marseille\\
CNRS - Universit\'{e} d'Aix-Marseille \\
Marseille 13453, France\\
$^*$E-mail: Valentin.Zagrebnov@univ-amu.fr}
%www.university\_name.edu}

%\vspace{0.5cm}

\noindent
\textit{Dedicated to Igor Volovich, friend, colleague and coauthor,}

\hspace{5cm} \textit{on the occasion of his 75th birthday.}

%\hspace{3cm}

\begin{abstract}
The purpose of the present notes is to examine the following issues related to the
the Chernoff estimate: (1) For contractions on a Banach space we modify
the $\sqrt{n}$-estimate and apply it in the proof of the Chernoff product formula
for $C_0$-semigroups in the \textit{strong} operator topology.
(2) We use the idea of a {probabilistic} approach, proving the
Chernoff estimate in the strong operator topology, to uplift it to the
\textit{operator-norm} estimate for \textit{quasi-sectorial} contraction semigroups.
(3) The operator-norm Chernoff estimate is applied to {quasi-sectorial} contraction semigroups
for proving the operator-norm convergence of the \textit{Dunford-Segal} approximants.
\end{abstract}


%\vspace{0.5cm}

%\noindent
%\textit{Dedicated to Igor Volovich, my friend, colleague and coauthor, on the occasion of
%his 75th birthday.}

\bodymatter
%%%%%%%%%%%%%%%%%%%%%%%%%%%%%%%%%%%%%%%%%%%%%%%%%%%%%%%%%%%%%%%%%%%%%%%%%%%%%%%%%%%%%%%%%%%%%%%%%
\section{Introduction}\label{sec:1}
\setcounter{equation}{0}
\renewcommand{\theequation}{\arabic{section}.\arabic{equation}}
%\setcounter{theo}{0}
%\renewcommand{\thetheo}{\arabic{section}.\arabic{theo}}
%%%%%%%%%%%%%%%%%%%%%%%%%%%%%%%%%%%%%%%%%%%%%%%%%%%%%%%%%%%%%%%%%%%%%%%%%%%%%%%%%%%%%%%%%%%%%%%%%
The Chernoff $\sqrt{n}$-Lemma, see Lemma 2 in Ref.\citenum{Cher68}, is known as a key tool
in the theory of semigroup approximations, see, for example, Ref.\citenum {EN00}
(Chapter III, Section 5) and Ref.\citenum{Zag20}. A large variety of applications
of the Chernoff approximation method (in the strong operator topology) one finds in a recent
survey Ref.\citenum{But20} (Section 2).
For the reader convenience and for motivation of the present notes we show this lemma in the
following below.

%%%%%%%%%%%%%%%%%%%%%%%%%%%%%%%%%%%%%%% Lemma %%%%%%%%%%%%%%%%%%%%%%%%%%%%%%%%%%%%%%%%%%%%%%%%%%%%
\begin{lemma}\label{lem:2.1.9}
Let bounded operator $C$ on a Banach space $\mathfrak{X}$
\emph{(}$C \in \mathcal{L}(\mathfrak{X})$\emph{)} be a contraction, that is, $\|C\| \leq 1$. Then
$\{e^{t \, (C - \mathds{1})}\}_{t\geq 0}$ is a norm-continuous contraction semigroup on
$\mathfrak{X}$ and one has the estimate
\begin{equation}\label{eq:2.1.10}
\|(C^n - e^{n \, (C-\mathds{1})})x\| \leq \ \sqrt{n} \ \|(C-\mathds{1})x\| \ \ ,
\end{equation}
for all $x\in\mathfrak{X} $ and natural $\ n \in \mathbb{N}$.
\end{lemma}
%%%%%%%%%%%%%%%%%%%%%%%%%%%%%%%%%%%%%%%%%%%%%%%%%%%%%%%%%%%%%%%%%%%%%%%%%%%%%%%%%%%%%%%%%%%%%%%%%%
\begin{proof}
To prove the inequality (\ref{eq:2.1.10}) we use the representation
\begin{equation}\label{eq:2.1.11}
C^n-e^{n(C-\mathds{1})} = e^{-n}\sum_{m=0}^\infty \frac{n^m}{m!}(C^n-C^m) \ .
\end{equation}
To proceed we insert
\begin{equation}\label{eq:2.1.12}
\|(C^n-C^m)x\|  \leq  \left\|(C^{|n-m|}  - \mathds{1})x \right\|\leq  |m-n|
\|(C- \mathds{1})x \| \ ,
\end{equation}
into (\ref{eq:2.1.11}) to obtain by the Cauchy-Schwarz inequality  the estimate:
\begin{equation}\label{eq:2.1.13}
\begin{split}
&\|(C^n - e^{n(C-\mathds{1})})x \| \leq \|(C - \mathds{1})x \| \
e^{-n} \ \sum_{m=0}^\infty \frac{n^m}{m!} |m-n| \leq\\
&  \{\sum_{m=0}^\infty  e^{-n} \  \frac{n^m}{m!} |m-n|^2\} ^{1/2} \|(C - \mathds{1})x\| \ ,
\ x\in\mathfrak{X} \ ,
\end{split}
\end{equation}
Note that the sum in the right-hand side of (\ref{eq:2.1.13}) can be calculated explicitly.
It is equal to $\sqrt{n}$, which yields (\ref{eq:2.1.10}).
%\hfill $\square$
\end{proof}
%%%%%%%%%%%%%%%%%%%%%%%%%%%%%%%%%%%%%%%%%%%%%%%%%%%%%%%%%%%%%%%%%%%%%%%%%%%%%%%%%%%%%%%%%%%%%%%%%%

The aim of the present notes is to scrutinise the following issues related to the
the Chernoff estimate.\\
First, we modify for contractions on a Banach space $\mathfrak{X}$ the $\sqrt{n}$-estimate
(\ref{eq:2.1.10}) and apply new estimates \textit{\`{a} la} Chernoff (see Section \ref{sec:2}
and Section \ref{sec:3}) in the proof of the Chernoff product formula for \textit{strongly}
continuous semigroups ($C_0$-semigroups) in the \textit{strong} operator topology, cf.
Ref.\citenum{But20}.\\
Second, we use the idea of the \textit{probabilistic} approach \cite{CaZ01,Zag17}, that proving the
Chernoff estimate in the strong operator topology (Section \ref{sec:2}), to uplift it to the
\textit{operator-norm} estimate for a special class of contractions: the \textit{quasi-sectorial}
contractions, see Section \ref{sec:4}.\\
Finally, in Section \ref{sec:5} we use the operator-norm Chernoff estimate
for illustration of its \textit{direct} application in the approximation theory of holomorphic
$C_0$-semigroups for $m$-sectorial generators (that is, for \textit{quasi-sectorial} contractions)
in a Hilbert space. This allows to prove, besides the
\textit{Euler} approximation formula, the operator-norm convergence with optimal rate of the
\textit{Dunford-Segal} approximants \cite{GoTo14}.

We warn the readers against a confusion between our probabilistic approach \cite{CaZ01,Zag17}
to alternative proof of the Chernoff estimate and a probabilistic approach to representation of
$C_0$-semigroups \cite{Pf93} exploited in Ref.\citenum{GoKoTo19} for developing the approximation
theory of operator semigroups.

%(Theorem \ref{th:6.2.3}

%%%%%%%%%%%%%%%%%%%%%%%%%%%%%%%%%%%%%%%%%%%%%%%%%%%%%%%%%%%%%%%%%%%%%%%%%%%%%%%%%%%%%%%%%%%%%%%%%
\section{The $\sqrt[3]{n}$-Lemma and Product Formul{\ae}}\label{sec:2}
%%%%%%%%%%%%%%%%%%%%%%%%%%%%%%%%%%%%%%%%%%%%%%%%%%%%%%%%%%%%%%%%%%%%%%%%%%%%%%%%%%%%%%%%%%%%%%%%%
We start by a technical lemma. It is a \textit{revised} version of the estimate (\ref{eq:2.1.10}).
%Chernoff $\sqrt{n}$-Lemma \ref{lem:2.1.9}.
Our \textit{variational} estimate (\ref{eq:1.8.51}) in
$\sqrt[3]{n}$-Lemma \ref{lem:1.8.23} and the \textit{probabilistic} approach are, in a certain
sense, more flexible than (\ref{eq:2.1.10}). Indeed, a revised scheme of the proof will be
used later (Section \ref{sec:4}) for \textit{uplifting} the convergence of the Chernoff and
the Lie-Trotter product formul{\ae} to the \textit{operator-norm} topology.
%%%%%%%%%%%%%%%%%%%%%%%%%%%%%%%%%%% Lemma %%%%%%%%%%%%%%%%%%%%%%%%%%%%%%%%%%%%%%%%%%%%%%%%%%%%%%%
\begin{lemma}\label{lem:1.8.23}\emph{($\sqrt[3]{n}\ $-Lemma)}
Let $C$ be a contraction on a Banach space $\mathfrak{X}$. Then
$\{e^{t (C - \mathds{1}}\}_{t\geq 0}$ is a norm-continuous
contraction semigroup on $\mathfrak{X}$ and one has the estimate
\begin{equation}\label{eq:1.8.51}
\|(C^n- e^{n\, (C-\mathds{1})})\, x\|\leq \frac{n}{\epsilon_n^2} \ 2\, \|x\| + \epsilon_n
\ \|(\mathds{1}-C)\, x\, \|\, , \quad \ n\in {\mathds{N}} \, ,
\end{equation}
for all $x\in \mathfrak{X}$ and $\epsilon_n > 0$. For the optimal value of the
parameter $\epsilon_n$ \emph{:}
\begin{equation}\label{eq:1.8.51-0}
\epsilon_{n}^* := \left(\frac{4\, n \, \|x\|}{\|(\mathds{1}-C)\, x\, \|}\right)^{1/3},
\end{equation}
in the right-hand side of \emph{(\ref{eq:1.8.51})} we obtain estimate
\begin{equation}\label{eq:1.8.51-1}
\|(C^n- e^{n\, (C-\mathds{1})})\,\|\leq
\frac{3}{2} \, \sqrt[3]{n} \ \|2 \ (\mathds{1}-C)\,\|^{2/3} \,  ,
\end{equation}
which we call the $\sqrt[3]{n}\ $-Lemma.
\end{lemma}
%%%%%%%%%%%%%%%%%%%%%%%%%%%%%%%%%%%%%%%%%%%%%%%%%%%%%%%%%%%%%%%%%%%%%%%%%%%%%%%%%%%%%%%%%%%%%%%%%
\begin{proof}
Since operator $C$ is bounded and $\|C\|\leq 1$, the operator $(\mathds{1}-C)$ is generator of a
norm-continuous contraction semigroup since:
\begin{equation}\label{eq:1.8.51-2}
\|\, e^{- t \,(\mathds{1}- C)}\, \| \leq e^{-t} \
\left\|\sum_{m=0}^{\infty} \frac{t^m}{m!} C^m \right\| \leq 1 \, , \quad t\geq 0 \, .
\end{equation}

For proving the estimate (\ref{eq:1.8.51}) we use representation
\begin{equation}\label{eq:1.8.52}
C^n-e^{n(C-\mathds{1})} = e^{-n} \ \sum_{m=0}^\infty \frac{n^m}{m!}\, (C^n-C^m)\, .
\end{equation}
%
%Let $\epsilon_n:=n^{\delta +1/2}$ for $n\geq 1$.
Then we \textit{split} the sum (\ref{eq:1.8.52})
into two parts: the \textit{central} part for $|m-n|\leq\epsilon_n$ and the \textit{tails} for
$|m-n|>\epsilon_n$. Optimisation of the \textit{splitting} parameter $\epsilon_n$ in
(\ref{eq:1.8.51}) yields the best estimate (\ref{eq:1.8.51-1}).

For evaluation the \textit{tails} we use the \textit{Tchebych\"{e}v inequality}. Let
$X_n \in \mathds{N}_0$
be the \textit{Poisson random variable} with the \textit{rate} parameter $n\in \mathds{N}$,
that is, with the
probability distribution $\mathbb{P}\{X_{n} = m\} = n^m e^{-n}/m! \, $. Then one gets for
expectation: ${\mathbb{E}}(X_{n})=n$, and for variance: $\mbox{Var}(X_{n}):=
{\mathbb{E}}((X_{n} - {\mathbb{E}}(X_{n}))^{2}) = n$. That being so, the Tchebych\"{e}v inequality
yields
\begin{equation}\label{eq:1.8.52-0}
 \ \ \mathbb{P}\{|X_{n} - {\mathbb{E}}(X_{n})|>\epsilon\} \leq
 \frac{\mbox{Var}(X_{n})}{\epsilon_{n}^2}\, ,
 \ \ {\rm{for \ any}\ } \ \epsilon_n >0 .
\end{equation}

Note that although for any $x \in \mathfrak{X}$ there is
an evident bound: $\|(C^n-C^m)\, x\| \leq 2 \, \|x\|$, when estimating (\ref{eq:1.8.52})
we shall also use below inequalities
\begin{equation}\label{eq:1.8.52-1}
\begin{split}
\|(C^{\,n}-C^{\,m})\, x\| & = \|C^{\, n-k}(C^{k}-C^{\, m-n+k})\, x\| \\
& \leq |m-n| \ \|C^{\, n-k}(\mathds{1}-C)\, x\| , \ \ \ k= 0,1,\ldots,n \ ,
\end{split}
\end{equation}
that keep \textit{difference}: $(\mathds{1}-C)\, x$.
%Put in this inequality $k = [\epsilon_n]$, here $[x]$ denotes the integer part of $x \geq 0$.
Then by $\|C\|\leq 1$ and by the Tchebych\"{e}v inequality (\ref{eq:1.8.52-0}) we obtain the
estimate for
\textit{tails}:
\begin{equation}\label{eq:1.8.53}
\begin{split}
&e^{-n}\sum_{|m-n|>\epsilon_n} \frac{n^m}{m!} \|(C^n-C^m)\, x\| \leq
e^{-n}\sum_{|m-n|>\epsilon_n} \frac{n^m}{m!} \cdot \ 2 \, \|x\| \, \\
& = \ \mathbb{P}\{|X_{n} - {\mathbb{E}}(X_{n})|>\epsilon_n\} \cdot \ 2\, \|x\| \leq
\frac{n}{\epsilon_n^2} \ 2\, \|x\| \, .
%= {2 \over n^{2 \delta}} \ \|u\| \, .
\end{split}
\end{equation}

To evaluate the \textit{central} part of the sum (\ref{eq:1.8.52}), when $|m-n|\leq\epsilon_n$,
note that by virtue of (\ref{eq:1.8.52-1}):
\begin{eqnarray}\label{eq:1.8.53-0}
\|(C^n-C^m)x\| &\leq & |m-n| \ \|C^{n-[\epsilon_n]}\, (\mathds{1}-C) x\| \\
& \leq & \epsilon_n \ \|(\mathds{1}-C) \, x\| . \nonumber
\end{eqnarray}
Then we obtain:
\begin{equation}\label{eq:1.8.53-1}
e^{-n} \, \sum_{|m-n|\leq\epsilon_n} \frac{n^m}{m!} \ \|(C^n-C^m)\, x\| \leq \epsilon_n \
\|(\mathds{1}-C)\, x\| \, , \quad x \in \mathfrak{X} \, ,
\end{equation}
for $n\in {\mathds{N}}$. Estimate (\ref{eq:1.8.53-1}), together with (\ref{eq:1.8.53}),
yield (\ref{eq:1.8.51}) for all $x\in \mathfrak{X}$ and $\epsilon_n > 0$.

Minimising the estimate (\ref{eq:1.8.51}) with respect to parameter $\epsilon_n > 0$ one gets
for $\epsilon_{n}$ the optimal value (\ref{eq:1.8.51-0}) and
\begin{equation}\label{eq:1.8.53-2}
\frac{n}{{\epsilon_{n}^{*}}^{2}} \ 2\, \|x\| + \epsilon_{n}^* \ \|(\mathds{1}-C)\, x\, \| =
\frac{3}{2} \, \sqrt[3]{n} \ (4\, \|x\|)^{1/3} \ \|(\mathds{1}-C)\, x\, \|^{2/3} \,  ,
\end{equation}
for all $x\in \mathfrak{X}$ and $n\in {\mathds{N}}$. As a consequence, (\ref{eq:1.8.51})
and (\ref{eq:1.8.53-2}) yield estimate (\ref{eq:1.8.51-1}), which is the $\sqrt[3]{n}\ $-Lemma.
%%%%%%%%%%%%%%%%%%%%%%%%%%%%%%%%%%%%%%%%%%%%%%%%%%%%%%%%%%%%%%%%%%%%%%%%%%%%%%%%%%%%%%%%%%%%%%%%%%%
\end{proof}
%%%%%%%%%%%%%%%%%%%%%%%%%%%%%%%%%%%%%%%%%%%%%%%%%%%%%%%%%%%%%%%%%%%%%%%%%%%%%%%%%%%%%%%%%%%%%%%%%%%
%Note that for $\delta= 0$ the estimate (\ref{eq:1.8.51}) gives for large $n$ the same asymptotic
%as the Chernoff
%$\sqrt{n}$-Lemma, whereas for optimal value $\delta= (-1/6)$ the asymptotic $2 \sqrt[3]{n}$ is
%better than
%(\ref{eq:2.1.10}). We call this result the $\sqrt[3]{n}$-Lemma.
%%%%%%%%%%%%%%%%%%%%%%%%%%%%%%%%%%%%%%%% Theorem %%%%%%%%%%%%%%%%%%%%%%%%%%%%%%%%%%%%%%%%%%%%%%%%%%
\begin{theorem}\label{th:1.8.24}\emph{(Chernoff product formula \cite{Cher68, Cher74})}
Let $\Phi: t\mapsto\Phi(t)$ be a function from $\mathbb{R}_{0}^+$ to contractions on $\mathfrak{X}$
such that $\Phi(0) = \mathds{1}$. Let $\{U_{A}(t)\}_{t\geq 0}$ be a contraction $C_0$-semigroup,
and let domain $D \subset {\rm{dom}}(A)$ be a core of related generator $A$.

If the function $\Phi(t)$ has a strong right-derivative $\Phi'(+0)$ at $t=0$
{\emph{(}}that is, $\Phi'(+0)x$ exists for any $x\in {\rm{dom}}(\Phi'(+0))${\emph{)}} and if
\begin{equation}\label{eq:1.8.55-0}
\Phi'(+0)\, x := \lim_{t\rightarrow +0} \frac{1}{t} (\Phi(t)- \mathds{1}) \, x = -\, A \, x \ ,
\end{equation}
for all $x\in D$, then
\begin{equation}\label{eq:1.8.55}
\lim_{n \rightarrow \infty} [\Phi(t/n)]^n \, x = U_{A}(t)\, x  \ ,
\end{equation}
for all $t\in \mathbb{R}_{0}^+$ and $x\in \mathfrak{X}$.
\end{theorem}
%%%%%%%%%%%%%%%%%%%%%%%%%%%%%%%%%%%%%%%%%%%%%%%%%%%%%%%%%%%%%%%%%%%%%%%%%%%%%%%%%%%%%%%%%%%%%%%%%%%
\begin{proof}
Consider the bounded approximations $\{A_n(s)\}_{n \geq 1}$ of generator $A$:
\begin{equation}\label{eq:1.8.56}
A_n(s) := \frac{\mathds{1} - \Phi(s/n)}{s/n} \ ,  \quad s \in {\mathds{R}}^{+}  \, ,
\quad n \in {\mathds{N}} \, .
\end{equation}
Note that these operators are $m$-accretive:
%$(A_n(s) + \zeta \mathds{1})^{-1} \in \mathcal{L}(\mathfrak{X})$ and
$\|(A_n(s) + \zeta \, \mathds{1})^{-1}\| \leq ({\rm{Re}} (\zeta))^{-1}$ for ${\rm{Re}}(\zeta) > 0$
and for any $n \in {\mathds{N}}$. By $\|\Phi(t)\|\leq 1$ together with (\ref{eq:1.8.56})
and condition (\ref{eq:1.8.55-0}) we obtain $\|e^{- t \, A_n(s)}\| \leq 1$, but also
\begin{equation}\label{eq:1.8.57}
\lim_{n \rightarrow \infty} A_n(s) \, x = A \, x \ ,
\end{equation}
for all $x \in D $ and any $s \in {\mathds{R}}^{+}$.
Then, given that $D = {\rm{core}}(A)$, by virtue of
the \textit{Trotter-Neveu-Kato} generalised strong convergence theorem (see, e.g.,
Ref.\citenum{Dav80} (Theorem 3.17), or Ref.\citenum{EN00} (Chapter III, Theorem 4.9)) one obtains
\begin{equation}\label{eq:1.8.58}
\lim_{n \rightarrow \infty} e^{- t \, A_n(s)} \, x = U_{A}(t)\, x \ , \quad x \in \mathfrak{X}\, ,
\quad s > 0 \, , \quad  t\in {\mathds{R}}_{0}^{+}\, .
\end{equation}
So, (\ref{eq:1.8.58}) is the strong and \textit{uniform} in $t$ and in $s$ convergence of
contractive approximants $\{e^{- t \, A_n(s)}\}_{n\geq 1}$ for $t \in {[0, \tau]}$ and
$s \in {(0, s_{0}]}$.

Now, by Lemma \ref{lem:1.8.23} (\ref{eq:1.8.51-1}) we obtain for contraction $C := \Phi(t/n)$:
\begin{equation}\label{eq:1.8.59}
\begin{split}
&\|[\Phi(t/n)]^n \, x - e^{- t \, A_n(t)} \, x\| = \|([\Phi(t/n)]^n - e^{{n}(\Phi(t/n)-
\mathds{1})}) \ x \| \\
&\leq \frac{3}{2} \, \sqrt[3]{n} \ (4\, \|x\|)^{1/3} \ \|(\mathds{1}-\Phi(t/n))\, x\, \|^{2/3}
\ , \quad x \in \mathfrak{X}\ .
\end{split}
\end{equation}
Since by (\ref{eq:1.8.57}) one gets for any $x\in D $ and uniformly on $(0,t_0]$:
\begin{equation}\label{eq:1.8.60}
\lim_{n \rightarrow \infty} \sqrt[3]{n} \ \|\, (\mathds{1}- \Phi(t/n)) \ x \,\|^{2/3} =
\lim_{n \rightarrow \infty} t^{2/3} \, n^{- 1/3} \ \|A_n(t) \ x\|^{2/3} = 0 \ ,
\end{equation}
equations (\ref{eq:1.8.59}) and (\ref{eq:1.8.60}) provide uniformly on $(0,t_0]$
\begin{equation}\label{eq:1.8.61}
\lim_{n \rightarrow \infty}\|\,[\Phi(t/n)]^n \, x - e^{- t \, A_n(t)} \, x\| = 0 , \
\quad x \in D \ .
\end{equation}
Then (\ref{eq:1.8.58}) and (\ref{eq:1.8.61}) yield uniformly in $t\in [0, t_0]$ the limit:
\begin{equation}\label{eq:1.8.55-1}
\lim_{n \rightarrow \infty} [\Phi(t/n)]^n \, x = U_{A}(t)\, x  \ , \quad x \in D \ .
\end{equation}

Note that by \textit{density} of $D$ and by the \textit{uniform} estimate
$\|\,[\Phi(t/n)]^n \, x - e^{- t \, A_n(t)}x\|\leq 2 \, \|x\|$ the convergence in (\ref{eq:1.8.61})
can be extended to all $x\in \mathfrak{X}$. Indeed, it is known that on the \textit{bounded}
subsets of $\mathcal{L}(\mathfrak{X})$ the topology of \textit{point-wise} convergence on a
\textit{dense} subset $D \subset\mathfrak{X}$ coincides with the \textit{strong} operator topology,
see, e.g., Ref.\citenum{Kato95} (Chapter III, Lemma 3.5).
As a consequence, the limit (\ref{eq:1.8.61}), when being extended to $x\in \mathfrak{X}$, and
limit (\ref{eq:1.8.58}) yield the extension of (\ref{eq:1.8.55-1}) to (\ref{eq:1.8.55}).
%\hfill $\square$
\end{proof}
%%%%%%%%%%%%%%%%%%%%%%%%%%%%%%%%%%%%%%%%%%%%%%%%%%%%%%%%%%%%%%%%%%%%%%%%%%%%%%%%%%%%%%%%%%%%%%%%%%

The limit (\ref{eq:1.8.55}) that involves derivative (\ref{eq:1.8.55-0}) is called the
\textit{Chernoff product formula} for contractive $C_0$-semigroup $\{U_{A}(t)\}_{t\geq 0}$
in the \textit{strong} operator topology, cf. Ref.\citenum{EN00} (Chapter III, Section 5a.).
%%%%%%%%%%%%%%%%%%%%%%%%%%%%%% Proposition  %%%%%%%%%%%%%%%%%%%%%%%%%%%%%%%%%%%%%%%%%%%%%%%%%%%%%%
\begin{proposition}\label{prop:1.8.25}{\emph{(Lie-Trotter product formula \cite{Trot59, Cher74})}}
Let $A$, $B$ and $C$ be generators of contraction $C_0$-semigroups on $\mathfrak{X}$. Suppose that
algebraic sum
\begin{equation}\label{eq:1.8.62}
C x = A x + B x \ ,
\end{equation}
is valid for all $x\in D$, where domain $D = {\rm{core}}\,(C)$. Then the semigroup
$\{U_{C}(t)\}_{t\geq 0}$ can be
approximated on $\mathfrak{X}$ in the strong operator topology by the
Lie-Trotter product formula:
\begin{equation}\label{eq:1.8.63}
e^{- t C} \, x = \lim_{n\rightarrow \infty} (e^{- t A/n} e^{- t  B/n} )^n \, x  \ , \  \
\ x\in\mathfrak{X} \ ,
\end{equation}
for all $t\in \mathbb{R}_{0}^{+}$ and $C := \overline{(A + B)}$ is closure of the
operator-sum in {\emph{(\ref{eq:1.8.62})}}.
\end{proposition}
%%%%%%%%%%%%%%%%%%%%%%%%%%%%%%%%%%%%%%%%%%%%%%%%%%%%%%%%%%%%%%%%%%%%%%%%%%%%%%%%%%%%%%%%%%%%%%%%%%%
\begin{proof}
Let us define the contraction $\mathbb{R}_{0}^{+}\ni t \mapsto \Phi(t)$, $\Phi(0)= \mathds{1}$, by
\begin{equation}\label{eq:1.8.64}
\Phi(t) := e^{- t A} e^{- t B} \ .
\end{equation}
Note that if $x \in D$, then derivative
\begin{equation}\label{eq:1.8.65}
\Phi'(+0)x = \lim_{t\rightarrow +0} \frac{1}{t} ( \Phi(t)-\mathds{1})\ x = -(A + B)\ x \ .
\end{equation}
Now we are in position to apply Theorem \ref{th:1.8.24}. This yields (\ref{eq:1.8.63}) for
generator $C := \overline{(A + B)}$.
\end{proof}
%%%%%%%%%%%%%%%%%%%%%%%%%%%%%%%%%%%%%%%%%%%%%%%%%%%%%%%%%%%%%%%%%%%%%%%%%%%%%%%%%%%%%%%%%%%%%%%%%%%%
%%%%%%%%%%%%%%%%%%%%%%%%%%%%%%%%%%%%%%%%%% Corollary %%%%%%%%%%%%%%%%%%%%%%%%%%%%%%%%%%%%%%%%%%%%%%%
\begin{corollary}\label{cor:1.8.26}
Extensions of the strongly convergent Lie-Trotter product formula of
Proposition {\emph{\ref{prop:1.8.25}}} to quasi-bounded and holomorphic semigroups follows
through verbatim.
\end{corollary}
%%%%%%%%%%%%%%%%%%%%%%%%%%%%%%%%% Section %%%%%%%%%%%%%%%%%%%%%%%%%%%%%%%%%%%%%%%%%%%%%%%%%%%%%%%%%%
\section{More Chernoff's Estimates}
\label{sec:3}
%%%%%%%%%%%%%%%%%%%%%%%%%%%%%%%%%%%%%%%%%%%%%%%%%%%%%%%%%%%%%%%%%%%%%%%%%%%%%%%%%%%%%%%%%%%%%%%%%%%%
In this section we show a one more Chernoff-type estimate (see (\ref{eq:1.1.1})), which is of a
\textit{different} nature than {variational} estimate (\ref{eq:1.8.51})
($\sqrt[3]{n}$-Lemma \ref{lem:1.8.23}).
In fact, it is a kind of \textit{improvement} of the original Chernoff estimate (\ref{eq:2.1.10})
($\sqrt{n}$-Lemma \ref{lem:2.1.9}), which is still restricted to convergence in the strong
operator topology.
%%%%%%%%%%%%%%%%%%%%%%%%%%%%%%%%%%%%%%% Lemma %%%%%%%%%%%%%%%%%%%%%%%%%%%%%%%%%%%%%%%%%%%%%%%%%%%%
\begin{theorem}\label{th:1.1.1}
Let $C \in \mathcal{L}(\mathfrak{X})$ be contraction on a Banach space $\mathfrak{X}$. Then
$\{e^{t (C - \mathds{1}}\}_{t\geq 0}$ is a norm-continuous contractive semigroup on
$\mathfrak{X}$ and the following estimate
\begin{equation}\label{eq:1.1.1}
\|(C^n - e^{n(C-\mathds{1})}) \, x\| \leq \  \frac{n}{2}\ \Big(\|(C-\mathds{1})^2 \, x\| +
\frac{e^2}{3}\, \|(C-\mathds{1})^3 \, x\|\Big) \ ,
\end{equation}
holds for all $\ n \in \mathbb{N}$ and $x \in \mathfrak{X}$.
\end{theorem}
%%%%%%%%%%%%%%%%%%%%%%%%%%%%%%%%%%%%%%%%%%%%%%%%%%%%%%%%%%%%%%%%%%%%%%%%%%%%%%%%%%%%%%%%%%%%%%%%%%
\begin{proof}
The first assertion is proven in Lemma \ref{lem:1.8.23}, see (\ref{eq:1.8.51-2}).

To prove inequality (\ref{eq:1.1.1}) we use the \textit{telescopic} representation:
\begin{equation}\label{eq:1.1.2}
C^n-e^{n(C-\mathds{1})} =
\sum_{k=0}^{n-1} \ C^{n-k-1}\, (C - e^{(C-\mathds{1})})\, e^{k(C-\mathds{1})} \ .
\end{equation}
To proceed we exploit that operator $C\in \mathcal{L}(\mathfrak{X})$ is bounded and therefore
\begin{equation}\label{eq:1.1.3}
C - e^{(C-\mathds{1})} =  - \, \frac{1}{2} \ (\mathds{1} - C)^2  -
(\mathds{1} - C)^3 \, \sum_{m=3}^{\infty} \, \frac{(-1)^m}{m!}\, (\mathds{1} - C)^{m-3} \, ,
\end{equation}
for the operator-norm convergent series. Hence, owing to  $\|C\| \leq 1$ one gets estimate
\begin{equation}\label{eq:1.1.4}
\Big\|\sum_{m=3}^{\infty} \, \frac{1}{m!}\, (\mathds{1} - C)^{m-3}\Big\| \leq
\frac{1}{6} \ e^{\|\mathds{1} - C\|} \leq \frac{e^2}{6} \, .
\end{equation}
Then on account of (\ref{eq:1.1.2}) - (\ref{eq:1.1.4}) and (\ref{eq:1.8.51-2})
we obtain inequality (\ref{eq:1.1.1}).
%\hfill $\square$
\end{proof}
%%%%%%%%%%%%%%%%%%%%%%%%%%%%%%%%%%%%%%%%%%%%%%%%%%%%%%%%%%%%%%%%%%%%%%%%%%%%%%%%%%%%%%%%%%%%%%%%%%%%%
%%%%%%%%%%%%%%%%%%%%%%%%%%%%%%%%%%%%%%% Corollary%%%%%%%%%%%%%%%%%%%%%%%%%%%%%%%%%%%%%%%%%%%%%%%%%%%%
\begin{corollary}\label{cor:1.1.1} \emph{(Chernoff product formula)}
Let $\Phi: t\mapsto\Phi(t)$ be a function from $\mathbb{R}_{0}^+$ to contractions on $\mathfrak{X}$
such that $\Phi(0) = \mathds{1}$, which satisfies conditions of Theorem \emph{\ref{th:1.8.24}}.
Then
\begin{equation}\label{eq:1.1.5}
\lim_{n \rightarrow \infty}\|([\Phi(t/n)]^n - e^{{n}(\Phi(t/n)- \mathds{1})}) \ x \| = 0 \, ,
\quad \ x\in\mathfrak{X} \, ,
\end{equation}
and as a result one gets the product formula \emph{(\ref{eq:1.8.55})}.
\end{corollary}
%%%%%%%%%%%%%%%%%%%%%%%%%%%%%%%%%%%%%%%%%%%%%%%%%%%%%%%%%%%%%%%%%%%%%%%%%%%%%%%%%%%%%%%%%%%%%%%%%%%%
\begin{proof}
On account of Theorem \ref{th:1.1.1} we obtain by (\ref{eq:1.1.1}) the estimate
\begin{eqnarray}\label{eq:1.1.6}
&&\|([\Phi(t/n)]^n - e^{{n}(\Phi(t/n)- \mathds{1})}) \ x \| \leq \\
&&\frac{t^2}{2 n}\ \Big(\Big\|\frac{n^2}{t^2}\, (\mathds{1} - \Phi(t/n))^2 \, x\Big\| +
\frac{2\, e^2}{3} \ \Big\|\frac{n^2}{t^2}(\mathds{1} - \Phi(t/n))^2 \, x\Big\|\Big)
, \quad \ x\in\mathfrak{X} \, . \nonumber
\end{eqnarray}
Note that by (\ref{eq:1.8.57}) for any $t \in {\mathds{R}}^{+}$ we have on the dense set
$D = {\rm{core}}(A)$:
\begin{equation}\label{eq:1.1.7}
\lim_{n \rightarrow \infty} \frac{n}{t}\, (\mathds{1} - \Phi(t/n)) \, x = A \, x \ ,
\quad x \in D \, .
\end{equation}
Given that generator $A$ of contractive $C_0$-semigroup is \textit{accretive},
the range of resolvent: ${\rm{ran }}((A + \zeta \mathds{1})^{-1}) = \mathfrak{X}$,
for ${\rm{Re}}(\zeta) > 0$. As a consequence Ref.\citenum{Kato95} (Chapter III, Problem 2.9
and Chapter IX, \S1.2), domain ${\rm{dom}}(A^2) \subset \rm{dom}(A)$ is \textit{dense} in
$\mathfrak{X}$ and limit (\ref{eq:1.1.7}) provides
\begin{equation}\label{eq:1.1.8}
\lim_{n \rightarrow \infty} (A_{n}(t))^2 \, x = A^2 \, x  \ ,
\quad x \in D \subset {\rm{dom}}(A^2) \, ,
\end{equation}
where $A_{n}(t) := ({t}/n)^{-1}\, (\mathds{1} - \Phi(t/n))$, cf. (\ref{eq:1.8.56}),
and $D = {\rm{core}}(A)$.

By virtue of estimate (\ref{eq:1.1.6}) and  (\ref{eq:1.1.8}) we obtain
\begin{equation}\label{eq:1.1.9}
\lim_{n \rightarrow \infty} \|([\Phi(t/n)]^n - e^{{n}(\Phi(t/n)- \mathds{1})}) \ x \| = 0
 \ , \quad x \in D \, .
\end{equation}
Then similarly to concluding arguments in Theorem {\ref{th:1.8.24}}, saying that on the
\textit{bounded} subsets of $\mathcal{L}(\mathfrak{X})$ the topology of \textit{point-wise}
convergence on a \textit{dense} subset $D \subset\mathfrak{X}$ coincides with the \textit{strong}
operator topology, the limit (\ref{eq:1.1.9}) can be extended
%(\textit{Extension Principle} and \textit{Uniform Boundedness Principle}
%see, e.g., \cite[Theorem 1.16 and Chapter III, \S 1.5]{Kato95})
to $x \in \mathfrak{X}$.

Now, given that $D = {\rm{core}}(A)$, by virtue of the \textit{Trotter-Neveu-Kato} theorem we
obtain the limit (\ref{eq:1.8.58}), and owing to (\ref{eq:1.1.9}) for $x \in \mathfrak{X}$, we
deduce the \textit{Chernoff product formula} (\ref{eq:1.8.55}).
\end{proof}
%%%%%%%%%%%%%%%%%%%%%%%%%%%%%%%%%%%%%%%%%%%%%%%%%%%%%%%%%%%%%%%%%%%%%%%%%%%%%%%%%%%%%%%%%%%%%%%%%%%%
%%%%%%%%%%%%%%%%%%%%%%%%%%%%%%%%%%%%%%%% Remark %%%%%%%%%%%%%%%%%%%%%%%%%%%%%%%%%%%%%%%%%%%%%%%%%%%%
\begin{remark}\label{rem:1.1.0}
Resuming the Chernoff $\sqrt{n}$-estimate (\ref{eq:2.1.10}), and its varieties: (\ref{eq:1.8.51})
and (\ref{eq:1.1.1}), we conclude that due to the terms with difference $\|(C-\mathds{1})\ x\|$
all of them contr\^{o}l only the \textit{strong} convergence of the product formulae.
By definition (\ref{eq:1.8.56}) the \textit{rates}: $R_n(t)$, of these converges
\textit{conditioned} to $x \in D$ have the following asymptotic form for $t>0$ and large
$n \in \mathds{N}$:\\
(a) For (\ref{eq:2.1.10}): $R_n(t) =  \| A_{n}(t)\, x\|/\sqrt{n} \,$. \\
(b) For (\ref{eq:1.8.51}): $R_n(t) = \| A_{n}(t)\, x\|^{2/3}/\sqrt[3]{n} \, $. \\
(c) For (\ref{eq:1.1.1}): $R_n(t) = \| A_{n}(t)^2 \, x\|/{n} \, $.
\end{remark}
%%%%%%%%%%%%%%%%%%%%%%%%%%%%%%%%%%%%%%%%%%%%%%%%%%%%%%%%%%%%%%%%%%%%%%%%%%%%%%%%%%%%%%%%%%%%%%%%%%%%
%%%%%%%%%%%%%%%%%%%%%%%%%%%%%%%%%%%%%%%% Remark %%%%%%%%%%%%%%%%%%%%%%%%%%%%%%%%%%%%%%%%%%%%%%%%%%%%
\begin{remark}\label{rem:1.1.1}
None of these three methods has an evident straightforward extension that could ensure
the \textit{operator-norm} convergence of the Chernoff product formula \cite{Zag20}.
In the next Section \ref{sec:4} we show that only a relatively more sophisticated method (cf.(b))
based on the Tchebych\"{e}v inequality (Section \ref{sec:2}) is, \textit{a fortiori}, sufficiently
accurate. Indeed, it allows an \textit{uplifting} of convergence the Chernoff product formula to the
operator-norm topology for \textit{quasi-sectorial} contractions on a Hilbert space.
\end{remark}
%%%%%%%%%%%%%%%%%%%%%%%%%%%%%%%%%%%%%%%%%%%%%%%%%%%%%%%%%%%%%%%%%%%%%%%%%%%%%%%%%%%%%%%%%%%%%%%%%%%%
%%%%%%%%%%%%%%%%%%%%%%%%%%%%%%%%% Section %%%%%%%%%%%%%%%%%%%%%%%%%%%%%%%%%%%%%%%%%%%%%%%%%%%%%%%%%%
\section{{Operator-Norm Chernoff Estimate}}\label{sec:4}
%%%%%%%%%%%%%%%%%%%%%%%%%%%%%%%%%%%%%%%%%%%%%%%%%%%%%%%%%%%%%%%%%%%%%%%%%%%%%%%%%%%%%%%%%%%%%%%%%%%%
%%%%%%%%%%%%%%%%%%%%%%%%%%%%%%%%%%%%%%% Definition %%%%%%%%%%%%%%%%%%%%%%%%%%%%%%%%%%%%%%%%%%%%%%%%%
\begin{definition}\label{def:2.1.2}\cite{CaZ01, Zag08}
Contraction $C$ on the Hilbert space $\mathfrak{H}$ is called \textit{quasi-sectorial} for
semi-angle $\alpha\in [0, \pi/2)$ with respect to the vertex at $z=1$, if its numerical
range $W(C)\subseteq D_{\alpha}$. Here the subset of complex plane
\begin{eqnarray}\label{eq:2.1.1}
&& D_{\alpha}:= \\
&&\{z\in {\mathbb{C}}: |z|\leq \sin \alpha\} \cup
\{z\in {\mathbb{C}}: |\arg (1-z)|\leq \alpha \ {\rm{and}}\ |z-1|\leq
\cos \alpha \} .\nonumber
\end{eqnarray}
\end{definition}
%%%%%%%%%%%%%%%%%%%%%%%%%%%%%%%%%%%%%%%%%%%%%%%%%%%%%%%%%%%%%%%%%%%%%%%%%%%%%%%%%%%%%%%%%%%%%%%%%%%%

We comment that $D_{\alpha} \subset D_{\pi/2} = {\mathbb{D}}$ (unit disc) and recall
that a \textit{genera}l contraction $C$ satisfies a weaker condition: $W(C)\subseteq {\mathbb{D}}$.

Note that if operator $C$ is a quasi-sectorial contraction, then $\mathds{1}- C$ is an
$m$-sectorial operator with vertex $z=0$ and semi-angle $\alpha$. Consequently, the numerical
range: $W(\mathds{1}- C)\subset \overline{S}_{\alpha}\, $, for the closure of sector
\begin{equation*}
{S}_{\alpha} := \{z\in {\mathbb{C}\setminus \{0\}}: |\arg z| < \alpha \} \, .
\end{equation*}
Then for operator $C$ the limits: $\alpha=0$ and $\alpha = \pi/2$,
correspond respectively to self-adjoint and to contraction operators, whereas for $\mathds{1}- C$
they give a non-negative self-adjoint and an $m$-accretive (bounded) operators.

%For $\lambda > 0$ the resolvent $(A + \lambda\,\mathds{1})^{-1}$ of an $m$-sectorial operator
%$A$, with {semi-angle} $\alpha\in [0, \alpha_0]$, $\alpha_0 < \pi/2$,
%and vertex at $z=0$, gives an example of the quasi-sectorial contraction.
%%%%%%%%%%%%%%%%%%%%%%%%%%%%%%%%%%%%%%%%%%%%%%%%%%%%%%%%%%%%%%%%%%%%%%%%%%%%%%%%%%%%%%%%%%%%%%%%%%%%
The \textit{resolvent} of an $m$-sectorial operator $A$, with {semi-angle} $\alpha\in [0, \alpha_0]$,
for some $\alpha_0 < \pi/2$ and vertex at $z=0$, that is, for $W(A)\subset \overline{S}_{\alpha}$,
provides the first non-trivial example of a quasi-sectorial contraction. Indeed, the following
assertion holds.
%
%%%%%%%%%%%%%%%%%%%%%%%%%%%%%%%%%%%%%%%%%%%%%% Proposition %%%%%%%%%%%%%%%%%%%%%%%%%%%%%%%%%%%%%%%%%
\begin{proposition}\label{prop:2.1.3}
Let $A$ be $m$-sectorial operator with {semi-angle} $\alpha\in [0, \pi/4]$ and vertex at $z=0$.
Then $\{F(t):= (\mathds{1} + t A)^{-1}\}_{\, t > 0}$ is a family
of quasi-sectorial contractions, such that numerical ranges $W(F(t)) \subset D_{\alpha}$
for all $t > 0$.
\end{proposition}
%%%%%%%%%%%%%%%%%%%%%%%%%%%%%%%%%%%%%%%%%%%%%%%%%%%%%%%%%%%%%%%%%%%%%%%%%%%%%%%%%%%%%%%%%%%%%%%%%%%%
\begin{proof}
Seeing that the spectrum $\sigma(A)$ is a subset of the closure $\overline{W(A)}$ of numerical
range $W(A)\subset \overline{S}_{\alpha}$, by the estimate of resolvent:
$\|(A - z \, \mathds{1})^{-1}\|\leq \big(\mbox{dist}(z, \overline{W(A)})\big)^{-1}$ and by
$\overline{W(A)}\subseteq \overline{S}_{\alpha}$ we obtain the operator-norm bound
\begin{equation} \label{eq:2.1.2}
\|F(t)\|\leq \frac{1}{t \ {\rm{dist}}(1/t \, , - \, \overline{S}_{\alpha})} \, = \, 1 \, ,
\quad t > 0 \, .
\end{equation}
As a consequence the family of operators $\{F(t)\}_{\, t \geq 0}$ consists of contractions with
numerical ranges $W(F(t)) \subset {\mathbb{D}}$.

Next, for any $u\in \mathfrak{H}$ ($\|u\| = 1$) one gets
$(u,F(t)u)= (v_t,v_t) + t (A v_t,v_t) \in \overline{S}_{\alpha}$, where $v_t := F (t)\,u $. So, for
all $t > 0$ numerical range $W(F(t)) \subseteq \overline{S}_{\alpha}$. Similarly, one finds
that $(u,(\mathds{1}-F(t))u) = t (v_t , A v_t) + t^2 (A v_t , A v_t) \in \overline{S}_{\alpha}$,
that is, $W(\mathds{1}-F(t)) \subseteq \overline{S}_{\alpha}$, or
$W(F (t)) \subseteq (1 - \overline{S}_{\alpha})$. Then for all $t > 0$:
\begin{equation*}
%\label{eq:2.1.3}
W(F (t)) \subseteq (\overline{S}_{\alpha} \cap (1 - \overline{S}_{\alpha})) \subset \mathbb{D} .
\end{equation*}
Moreover, by Definition \ref{def:2.1.2} the condition $\alpha \leq \pi/4$ yields that
$(\overline{S}_{\alpha} \cap (1 - \overline{S}_{\alpha}))\subset D_{\alpha}$. Hence, for
$\alpha\in [0, \pi/4]$ the operators $\{F(t)\}_{\, t \geq 0}$ are quasi-sectorial contractions
with numerical range in $D_{\alpha}$.
%\hfill $\square$
\end{proof}
%%%%%%%%%%%%%%%%%%%%%%%%%%%%%%%%%%%%%%%%%%%%%%%%%%%%%%%%%%%%%%%%%%%%%%%%%%%%%%%%%%%%%%%%%%%%%%%%%%%%

Note that the upper bound $\alpha \leq \pi/4$ is stemming from Definition \ref{def:2.1.2}
and observation that
$(\overline{S}_{\alpha} \cap (1 - \overline{S}_{\alpha}))\nsubseteq D_{\alpha}$
for $\alpha > \pi/4$, cf. (\ref{eq:2.1.1}).
%%%%%%%%%%%%%%%%%%%%%%%%%%%%%%%%%%%%%%%%%%%%%%%%%%%%%%%%%%%%%%%%%%%%%%%%%%%%%%%%%%%%%%%%%%%%%%%%%%
\begin{corollary}\label{cor:2.1.1}
Let $A$ be an {$m$-sectorial} operator with  semi-angle $\alpha\in[0, \pi/4]$ and with  vertex at
$z=0$. Then $\{e^{-t \, A}\}_{\, t \geq 0}$ is a holomorphic quasi-sectorial contraction semigroup
with numerical ranges $W(e^{-t \, A}) \subset D_{\alpha}$ for all $t > 0$ and one has the
strongly convergent Euler limit\emph{:}
\begin{equation}\label{eq:2.1.4}
s-\lim_{n\rightarrow\infty}(\mathds{1} + t A)^{-n} = e^{-t A} \, ,  \quad t\geq 0 \, .
\end{equation}
\end{corollary}
%%%%%%%%%%%%%%%%%%%%%%%%%%%%%%%%%%%%%%%%%%%%%%%%%%%%%%%%%%%%%%%%%%%%%%%%%%%%%%%%%%%%%%%%%%%%%%%%%
%%%%%%%%%%%%%%%%%%%%%%%%%%%%%%%%%%%%% Remark %%%%%%%%%%%%%%%%%%%%%%%%%%%%%%%%%%%%%%%%%%%%%%%%%%%%
\begin{remark}\label{rem:2.1.1} (Sketch of the proof.)
We comment that holomorphic property of $\{e^{- z \, A}\}_{\, z \in {S}_{\pi/2 -\alpha}}$
follows from conditions on generator $A$. Since $A$ is {$m$-sectorial} with  vertex at
$z=0$, it is \textit{a fortiori} accretive. Then by standard arguments for construction
of $C_0$-semigroups (see Ref.\citenum{Kato95}, Chapter IX) yield, due to (\ref{eq:2.1.2})
for approximants $\{(\mathds{1} + t\, A/n)^{-n}\}_{t \geq 0, \, n\in \mathbb{N}}\, $,
the strongly convergent \textit{Euler} formula (\ref{eq:2.1.4}).
Note that although by Proposition \ref{prop:2.1.3} the family $\{(\mathds{1} + t\, A)^{-1}\}_{t > 0}$
for $\alpha\in[0, \pi/4]$ consists of quasi-sectorial contractions with numerical ranges in
$D_{\alpha}$, a proof of the claim about \textit{inheritance} of this property by approximants
$\{(\mathds{1} + t\, A/n)^{-n}\}_{t > 0, n \in \mathbb{N}}$ and
by the limit $\{e^{- t \, A}\}_{t > 0}$ demands additional reasoning \cite{Zag08}.
It is heavily based on the \textit{Kato numerical range mapping} theorem Ref.\citenum{Kato65}.
\end{remark}
%%%%%%%%%%%%%%%%%%%%%%%%%%%%%%%%%%%%%%%%%%%%%%%%%%%%%%%%%%%%%%%%%%%%%%%%%%%%%%%%%%%%%%%%%%%%%%%%%%

We also note that \textit{extension} of Corollary \ref{cor:2.1.1} to semi-angle
$\alpha \in [0, \pi/2)$ needs merely a more refined arguments, which were developed
in Ref.\citenum{ArZ10}.
%%%%%%%%%%%%%%%%%%%%%%%%%%%%%%%%%%%%%%%%%%%%%%%%%%%%%%%%%%%%%%%%%%%%%%%%%%%%%%%%%%%%%%%%%%%%%%%%%%%%
%%%%%%%%%%%%%%%%%%%%%%%%%%%%% Proposition %%%%%%%%%%%%%%%%%%%%%%%%%%%%%%%%%%%%%%%%%%%%%%%%%%%%%%%%%%
\begin{proposition}\label{prop:2.1.10}
%{\emph{\cite{CaZ01, Zag08}}}
Let operator $C$ on a Hilbert space $\mathfrak{H}$ be a quasi-sectorial contraction with
semi-angle $0\leq\alpha < \pi/2$. Then for $\alpha< \alpha'< \pi/2$
\begin{equation}\label{eq:2.1.14}
\|C^n (\mathds{1}-C)\|\leq \frac{K_{\alpha, \alpha'}}{n+1} \ , \ n\in{\mathbb{N}} \ ,
\end{equation}
where $K_{\alpha, \alpha'}$ is given by \emph{(\ref{eq:6.2.4})}.
\end{proposition}
%%%%%%%%%%%%%%%%%%%%%%%%%%%%%%%%%%%%%%%%%%%%%%%%%%%%%%%%%%%%%%%%%%%%%%%%%%%%%%%%%%%%%%%%%%%%%%%%%%
\begin{proof}
Since operator $C$ is a quasi-sectorial contraction, the spectrum $\sigma(C)$ is a subset of
closure $\overline{W(C)}$ of the numerical range $W(C) \subset D_{\alpha}$. So, taking
$\alpha< \alpha'< \pi/2$ one gets by definition (\ref{eq:2.1.1}): $D_{\alpha'} \supset D_{\alpha}$.
Hence, contour $\partial D_{\alpha'}$ is outside of $D_\alpha$, but inside the unit disc
${\mathbb{D}}$, and all of them have only one common point $z = 1$. Then the
\textit{Riesz-Dunford} functional calculus provides the following representation of operator
in (\ref{eq:2.1.14}):
\begin{equation}\label{eq:6.2.2}
C^n (\mathds{1}-C) = \frac{1}{2\pi i} \int_{\partial D_{\alpha'}}  dz \
\frac{z^n \, (1-z)}{z \, \mathds{1}-C} \ .
\end{equation}
If $z$ belongs to resolvent set $\rho(C)$ of $C$, then evaluation of the norm of resolvent:
$\|(C - z\mathds{1})^{-1}\|\leq \mbox{dist}(z, \overline{W(C)})^{-1}$,
yields $\|(z\mathds{1}-C)^{-1}\|\leq \mbox{dist}(z,D_\alpha)^{-1}$ for
$z\in \partial D_{\alpha'} \subset \rho(C)$ in (\ref{eq:6.2.2}).
We consider the following parametrisation of the positively oriented contour
$\partial D_{\alpha'}$:

- for the arc$(A,B)$ with end-points at $A = e^{i (\pi/2-\alpha')}\sin\alpha'$
and at $B= e^{i (3\pi/2 + \alpha')}\sin\alpha'$, we take
$z(t)= e^{it}\sin\alpha'$ with $\pi/2-\alpha'\leq t\leq 3\pi/2+\alpha'$ ;

- for the straight lines $(1,A)$ and $(B,1)$, we take correspondingly
$z_{-}(s)=1 - s e^{- i\alpha'}$ with $s \in [0, \cos\alpha']$ and
$z_{+}(s)=1 - s e^{+ i\alpha'}$ with $s \in [\cos\alpha', 0]$.

As a consequence, by definition of the (\textit{shortest}) distance \textit{from}
$z \in \partial D_{\alpha'}$ \textit{to} $D_\alpha$, denoted as $\mbox{dist}(z,D_\alpha)$,
we obtain:

- $\|(z\mathds{1}-C)^{-1}\| \leq (\cos\alpha'\sin(\alpha' - \alpha))^{-1}$ for
$|\arg z|\geq\pi/2-\alpha'$, where we used that
$\mbox{dist}(z , D_\alpha)\in [\cos\alpha'\sin(\alpha' - \alpha), \,
(\sin\alpha' - \sin\alpha)]$ for $z \in {\rm{arc}}(A,B)$, that is, for
$z\in \{e^{it}\sin\alpha'\}_{t \in [\pi/2-\alpha', \ 3\pi/2+\alpha']}$ ;

- $\|(z\mathds{1}-C)^{-1}\| \leq (|1-z|\sin(\alpha'-\alpha))^{-1}$ for
$|\arg z|\leq \pi/2-\alpha'$, that is, for
$z\in \{1 - s e^{\mp i\alpha'}\}_{s\in [0,\ \cos\alpha']}$.
%We consider the following parametrisation of $\partial D_{\alpha'}$: for the arc $AB$, we take
%$z(t)= e^{it}\sin\alpha'$ with $\pi/2-\alpha'\leq t\leq 3\pi/2+\alpha'$, and for the straight lines
%$(1,A)$, $(1,B)$,
%we put correspondingly $z_{\mp}(s)=1 - s e^{\mp i\alpha'}$ with $0\leq s\leq\cos\alpha'$.

Then operator-norm estimate of the left-hand side in representation (\ref{eq:6.2.2}) takes
the form
\begin{eqnarray}
\|C^n (\mathds{1}-C)\| & \leq & \frac{1}{2\pi} \int_{\frac{\pi}{2}-\alpha'}^{{3\pi}/{2}+
\alpha'}\!\!  dt \ {|\sin\alpha'|^{n+1} |1- e^{it}\sin\alpha'|\over \cos\alpha'\sin(\alpha'-
\alpha)} + \nonumber\\
& & + {1\over\pi} \int_0^{\cos\alpha'}  ds \ {|(1- e^{i\alpha'}s)^n  e^{i\alpha'}s|\over
s \, \sin(\alpha'-\alpha)} \label{eq:6.2.3}\\
&\hspace{-4cm} \leq & \hspace{-2cm}
{2(\sin\alpha')^{n+1}\over\cos\alpha'\sin(\alpha'- \alpha)} +
\int_0^{\cos\alpha'} ds \ {\left((1-s\cos\alpha')^2 + s^2(\sin\alpha')^2\right)^{n/2}
\over\pi\sin(\alpha'-\alpha)} \ .\nonumber
\end{eqnarray}
Taking into account convexity of the mapping: $s \mapsto (1-s\cos\alpha')^2 + s^2(\sin\alpha')^2$,
for $s\in [0, \cos\alpha']$, one gets that
\begin{equation*}
(1-s\cos\alpha')^2 + s^2(\sin\alpha')^2\leq 1-s\cos\alpha' \, ,
\end{equation*}
which leads to inequality:
\begin{eqnarray*}
&&\int_0^{\cos\alpha'}ds \, \left((1-s\cos\alpha')^2 + s^2(\sin\alpha')^2\right)^{n/2}
 \leq  \int_0^{\cos\alpha'}ds \ (1-s\cos\alpha')^{n/2} \\
&& = \int_{(\sin\alpha')^2}^1 du \  {u^{n/2}\over\cos\alpha'}
\ \leq \ {1-(\sin\alpha')^{n+2}\over (n/2+1)\cos\alpha'} \ .
\end{eqnarray*}
Therefore, by (\ref{eq:6.2.3}) we obtain the estimate
\begin{eqnarray}\label{eq:6.2.4-1}
&&\|C^n (\mathds{1}-C)\| \leq {2 \ (\sin\alpha')^{n+1}\over\cos\alpha' \, \sin(\alpha'-\alpha)} +
2 \ {1-(\sin\alpha')^{n+2}\over \pi \, (n+2)\, \cos\alpha' \, \sin(\alpha'-\alpha)}
\leq \nonumber \\
&& {2 \ \over(n+1)\, \cos\alpha' \, \sin(\alpha'-\alpha)}\ \left(\frac{1}{\pi} +
(n+1)\ (\sin\alpha')^{n+1}\right) \ .
\end{eqnarray}

After \textit{optimisation} of the last factor in the right-hand side of (\ref{eq:6.2.4-1})
with respect to $n \in \mathbb{N}$, we infer (\ref{eq:2.1.14}) for
\begin{equation}\label{eq:6.2.4}
K_{\alpha, \alpha'}:={2\over\cos\alpha'\sin(\alpha'-\alpha)} \
\left({1\over\pi} - {1\over e\ln(\sin\alpha')}\right) \ ,
\end{equation}
where $\alpha< \alpha'< \pi/2 \, $.
%\hfill $\square$ \bigskip ${K_{\alpha, \alpha'} \over n+1}$ (\ref{eq:2.1.14})
\end{proof}
%%%%%%%%%%%%%%%%%%%%%%%%%%%%%%%%%%%%%%%%%%%%%%%%%%%%%%%%%%%%%%%%%%%%%%%%%%%%%%%%%%%%%%%%%%%%%%%%%%%

The property (\ref{eq:2.1.14}) implies that the quasi-sectorial contractions belong to the class
of so-called {\textit{Ritt's} operators} \cite{Ri53}. This allows to go beyond the
$\sqrt[3]{n}\ $-Lemma \ref{lem:1.8.23} to the $(\sqrt[3]{n})^{-1}$-Theorem and also
from estimates in the strong operator topology to the operator-norm topology.
The first step is the operator-norm Chernoff estimate (cf. (\ref{eq:1.8.51})):

%%%%%%%%%%%%%%%%%%%%%%%%%%%%%%%%%%%%%%%%%%%%%% Theorem %%%%%%%%%%%%%%%%%%%%%%%%%%%%%%%%%%%%%%%%%%%%
\begin{theorem}\label{th:6.2.2}{\emph{($(\sqrt[3]{n})^{-1}$-Theorem)}}
Let $C$ be a quasi-sectorial contraction on $\mathfrak{H}$ with
numerical range $W(C)\subseteq D_\alpha$, $0\leq \alpha <\pi/2$. Then
\begin{equation}\label{eq:6.2.5}
\left\|C^n - e^{n(C-\mathds{1})}\right\| \leq {L_\alpha \over n^{1/3}} \ , \ \ n \in \mathbb{N} \, ,
\end{equation}
where $L_\alpha=2K_\alpha+2$ and $K_\alpha :=\min_{\alpha' \in(\alpha, \pi/2)}K_{\alpha, \alpha'}$,
is defined by \emph{(\ref{eq:6.2.4})}.
\end{theorem}
%%%%%%%%%%%%%%%%%%%%%%%%%%%%%%%%%%%%%%%%%%%%%%%%%%%%%%%%%%%%%%%%%%%%%%%%%%%%%%%%%%%%%%%%%%%%%%%%%%%%%
\begin{proof}
With help of inequality (\ref{eq:2.1.14}) we can improve the
estimate of the \textit{central} part of the sum (\ref{eq:1.8.52}) in Lemma \ref{lem:1.8.23}.
Note that on account of (\ref{eq:1.8.52-1}) we obtain by (\ref{eq:2.1.14}) and $\|C\| \leq 1$:
\begin{eqnarray}\label{eq:2.1.15-1}
\|C^n-C^m\| \leq |m-n| \, \|C^{n-[\epsilon_n]}(\mathds{1}-C)\| \leq \epsilon_n \
\frac{K_\alpha}{n-[\epsilon_n]+1} \ ,
\end{eqnarray}
cf. (\ref{eq:1.8.53-0}). Here $\epsilon_n := n^{\delta + 1/2}$ for $\delta < 1/2$,
which makes sense for the estimate (\ref{eq:1.8.53}) of \textit{tails}, and $[\epsilon_n]$
is the \textit{integer} part of $\epsilon_n \geq |m - n|$. Then owing to (\ref{eq:2.1.15-1}) the
\textit{central} part has estimate
\begin{equation}\label{eq:2.1.15-2}
e^{-n} \, \sum_{|m-n|\leq\epsilon_n} \frac{n^m}{m!} \ \|(C^n-C^m)\, x\| \leq
\epsilon_n \ \frac{K_\alpha}{n-[\epsilon_n]+1} \, \| x\| \, , \quad x \in \mathfrak{X} \, ,
\quad n\in {\mathds{N}} \, .
\end{equation}
As a consequence, (\ref{eq:1.8.53}) and (\ref{eq:2.1.15-2}) yield instead of (\ref{eq:1.8.51-1})
(or (\ref{eq:2.1.10})) the \textit{operator-norm} estimate:
\begin{equation}\label{eq:2.1.15-3}
\left\|C^n - e^{n(C-\mathds{1})}\right\| \leq \frac{2}{n^{2\delta}} +
\epsilon_n \ \frac{K_\alpha}{n-[\epsilon_n]+1}
\ \  , \ \ n \in \mathbb{N} \ .
\end{equation}

Let $n_{0}\in \mathbb{N}$ satisfy inequality: $n_{0} \geq 2\,([\epsilon_n]-1)$.
Then (\ref{eq:2.1.15-3}) yields
\begin{equation}\label{eq:2.1.15}
\left\|C^n - e^{n(C-\mathds{1})}\right\| \leq \frac{2}{n^{2\delta}} +
\frac{2\,K_\alpha}{n^{1/2 -\delta}} \ , \quad n > n_{0} \ .
\end{equation}
Then estimate ${M_\alpha}/{n^{1/3}}$ of the Theorem  \ref{th:6.2.2} results from the
\textit{optimal} choice in (\ref{eq:2.1.15}) of the value: $\delta = 1/6 $.
\end{proof}
%%%%%%%%%%%%%%%%%%%%%%%%%%%%%%%%%%%%%%%%%%%%%%%%%%%%%%%%%%%%%%%%%%%%%%%%%%%%%%%%%%%%%%%%%%%%%%%%%%%
%%%%%%%%%%%%%%%%%%%%%%%%%%%%%%%%%%%%%%%%%%%%%%% corollary %%%%%%%%%%%%%%%%%%%%%%%%%%%%%%%%%%%%%%%%%
\begin{corollary}\label{cor:2.1.2}
If in \emph{(\ref{eq:2.1.15})} \textit{no} estimate of the convergence rate is required,
then the operator-norm convergence to zero in \emph{(\ref{eq:6.2.5})} follows directly
from the \textit{Riesz-Dunford} representation of $\ C^n - e^{n(C-\mathds{1})}$ as the
operator-valued integral along the contour $\partial D_{\alpha' \, }$ for
$\alpha< \alpha'< \pi/2$ \emph{:}
\begin{equation*}
C^n -e^{n(C-\mathds{1})} = {1\over 2\pi i} \int_{\partial D_{\alpha'}}
\hspace{-0.3cm}\, dz \ {z^n -e^{n(z-1)} \over z-C} \ ,  \quad n\in \mathbb{N} \ ,
\end{equation*}
which provides for $n \in \mathbb{N}$ inequalities:
\begin{eqnarray}
\|C^n -e^{n(C-\mathds{1})}\| & \leq & {1\over 2\pi} \int_{{\pi\over 2}-\alpha'}^{{3\pi\over 2}+
\alpha'}\hspace{-0.3cm}\, dt \ \sin\alpha' \  {|(e^{i\, t} \sin\alpha')^{n} -
\exp \{n \ (e^{i\, t}\sin\alpha'-1)\}|\over \cos\alpha' \ \sin(\alpha'-\alpha)}
\nonumber\\
&& + {1\over\pi} \int_0^{\cos\alpha'} \hspace{-0.3cm}\, ds \ {|(1-e^{-i\alpha'} s)^n -
\exp(-n \, s \, e^{-i\alpha'})
|\over s \ \sin(\alpha'-\alpha)}  \leq \nonumber
\\
&\hspace{-4cm} \leq & \hspace{-2cm}
{{\sin\alpha' \ {(\sin\alpha')^{n} + e^{n(\sin\alpha'-1)} \over \cos\alpha' \ \sin(\alpha'-
\alpha)}}} + \label{eq:6.2.8-1} \\
&& + \int_0^{n \cos\alpha'}\hspace{-0.3cm}\, dr \ {|(1-e^{-i\alpha'} r/n)^n -
\exp(- r \, e^{-i\alpha'})|\over r \ \sin(\alpha'-\alpha)}
\ . \nonumber
\end{eqnarray}
\emph{(}For parametrisation of integrands in representation \emph{(\ref{eq:6.2.8-1})} see
notes in the proof of Proposition \emph{\ref{prop:2.1.10}}.\emph{)}

Then $\lim_{n\rightarrow\infty}\, \|C^n -e^{n(C-\mathds{1})}\| = 0$ issues by conditions
$\alpha< \alpha'< \pi/2$ and the \textit{Lebesgue dominated convergence} theorem applied to
the last integral in the right-hand side of inequalities \emph{(\ref{eq:6.2.8-1})}.
\end{corollary}
%%%%%%%%%%%%%%%%%%%%%%%%%%%%%%%%%%%%%%%%%%%%%%%%%%%%%%%%%%%%%%%%%%%%%%%%%%%%%%%%%%%%%%%%%%%%%%%%%%%%
%{\red{To deduce in (\ref{eq:6.2.8-1}) the convergence rate estimate one needs more (refined)
%arguments, see ...}}
%%%%%%%%%%%%%%%%%%%%%%%%%%%%%%%%%%%%%%%%%%%%%%%%%%%%%%%%%%%%%%%%%%%%%%%%%%%%%%%%%%%%%%%%%%%%%%%%%%%
\begin{remark}\label{rem:6.2.2-2}\rm{
Recall \cite{CaZ01} that if the quasi-sectorial contraction $C$ is \textit{self-adjoint}
(i.e., $\alpha=0$), then one obtains for the rate of convergence optimal estimates:
\begin{equation}\label{eq:6.2.8}
\|C^n (\mathds{1}-C)\|\leq {1\over n+1}\ \ \mbox{and}\ \ \|C^n -
e^{n(C-\mathds{1})}\|\leq {e^{-1}\over n} \ , \quad n\in \mathbb{N} \ ,
\end{equation}
directly from the spectral representation of $C$.
}
\end{remark}
%%%%%%%%%%%%%%%%%%%%%%%%%%%%%%%%%%%%%%%%%%%%%%%%%%%%%%%%%%%%%%%%%%%%%%%%%%%%%%%%%%%%%%%%%%%%%%%%%%%%

In a full similarity with $(\sqrt[3]{n})$-Lemma for the \textit{strong} operator approximation,
the $(\sqrt[3]{n})^{-1}$-Theorem is only the first step in developing
the \textit{operator-norm} approximation formula \`{a} la Chernoff. To this end one needs an
operator-norm analogue of Theorem \ref{th:1.8.24}.
The preceding includes the Trotter-Neveu-Kato \textit{strong} convergence theorem. On
that account, now we need the \textit{operator-norm} extension of this assertion for
quasi-sectorial contractions.
%%%%%%%%%%%%%%%%%%%%%%%%%%%%%%%%%%%% Proposition %%%%%%%%%%%%%%%%%%%%%%%%%%%%%%%%%%%%%%%%%%%%%%%%%%%
\begin{proposition}\label{prop:6.3.1}{\emph{(Refs.\citenum{CaZ01, Zag20})}}
Let $\{X(s)\}_{s>0}$ be a family of $m$-sectorial operators in a Hilbert space $\mathfrak{H}$ such
that for some $0< \alpha <\pi/2$ and any $s>0$ the numerical range $W(X(s))\subseteq S_\alpha$.
Let $X_0$ be an $m$-sectorial operator defined in ${\mathfrak{H}}$, with $W(X_0)\subseteq S_\alpha$.
Then the two following assertions are equivalent \emph{:}
\begin{eqnarray*}
(a) & & \lim_{s\rightarrow +0} \left\|(\zeta \mathds{1} +X(s))^{-1} -
(\zeta \mathds{1} +X_0)^{-1}\right\| = 0 \ , \ \mbox{ for } \ \zeta\in S_{\pi-\alpha} \ ,  \\
(b) & & \lim_{s\rightarrow +0} \left\|e^{-tX(s)} - e^{-tX_0}\right\| = 0 \ , \
\mbox{ for } \ t> 0 \ .
\end{eqnarray*}
Here $S_\alpha = \{z\in \mathbb{C}: |\arg(z)| < \alpha\}$ is a sector in complex plane
$\mathbb{C}$ with semi-angle $\alpha$ and vertex at $z = 0$.
\end{proposition}
%%%%%%%%%%%%%%%%%%%%%%%%%%%%%%%%%%%%%%%%%%%%%%%%%%%%%%%%%%%%%%%%%%%%%%%%%%%%%%%%%%%%%%%%%%%%%%%%%%%%

Now $(\sqrt[3]{n})^{-1}$-Theorem \ref{th:6.2.2} (or Corollary \ref{cor:2.1.2}) \textit{and}
the Trotter-Neveu-Kato theorem (Proposition \ref{prop:6.3.1}) yield a desired
generalisation of the Chernoff product formula (cf. (\ref{eq:1.8.55})) for the operator-norm
convergence.
%%%%%%%%%%%%%%%%%%%%%%%%%%%%% Proposition %%%%%%%%%%%%%%%%%%%%%%%%%%%%%%%%%%%%%%%%%%%%%%%%%%%%%%%%%%
\begin{proposition}\label{prop:2.1.12}{\emph{(Refs.\citenum{CaZ01, Zag08})}}
Let $\{\Phi(s)\}_{s\geq 0}$ be a strongly measurable family of uniformly
quasi-sectorial contractions on a Hilbert space $\mathfrak{H}$, such that $\Phi(0) = \mathds{1}$
and  $W(\Phi(s)) \subset D_\alpha$ for all $s > 0$, where $0 \leq \alpha<\pi/2$. Let
\begin{equation}\label{eq:2.1.16}
X(s):=(\mathds{1}-\Phi(s))/s  \ , \quad s > 0 \, ,
\end{equation}
and let $X_0$ be a closed operator with non-empty resolvent set, defined in
$\mathfrak{H}$. Then the family $\{X(s)\}_{s>0}$ converges,
when $s\rightarrow +0$, in the uniform resolvent sense to the operator $X_0$ \emph{(}cf.
Proposition \emph{\ref{prop:6.3.1})}, if and only if
\begin{equation}\label{eq:2.1.17}
\lim_{n\rightarrow \infty} \left\|\Phi(t/n)^n -e^{-tX_0}\right\| = 0 \ , \ \ \ {\rm{for}}
\ \ t>0 \ .
\end{equation}
%
%Here $P_0$ denotes the orthogonal projection onto the subspace ${\mathfrak{H}}_0$.
\end{proposition}
%%%%%%%%%%%%%%%%%%%%%%%%%%%%%%%%%%%%%%%%%%%%%%%%%%%%%%%%%%%%%%%%%%%%%%%%%%%%%%%%%%%%%%%%%%%%%%%%%%%%%
\begin{corollary}\label{cor:4.1.1}\emph{(operator-norm Euler formula)}
If $A$ is an $m$-sectorial operator in Hilbert space $\mathfrak{H}$, with semi-angle
$\alpha\in [0,\pi/2)$ and vertex at $z=0$, then
\begin{equation}\label{eq:E1}
\lim_{n\rightarrow \infty} \left\|(\mathds{1}+tA/n)^{-n} - e^{-tA}\right\| = 0 \, ,
%\leq \frac{L_\alpha}{n} \, ,
\quad t\in S_{\pi/2-\alpha} \, ,
\end{equation}
for $n \in \mathbb{N}$.
\end{corollary}
%%%%%%%%%%%%%%%%%%%%%%%%%%%%%%%%%%%%%%%%%%%%%%%%%%%%%%%%%%%%%%%%%%%%%%%%%%%%%%%%%%%%%%%%%%%%%%%%%%%%
\begin{proof}
Since by condition the numerical range $W(A)\subset S_\alpha$, on account of Proposition
\ref{prop:2.1.3} and Remark \ref{rem:2.1.1}, $\{\Phi(t) := (\mathds{1}+tA)^{-1}\}_{t > 0}$ is a
family of quasi-sectorial contractions with $W(\Phi(t))\subset D_\alpha$, $\alpha\in [0,\pi/2)$.

Let $X(s):= (\mathds{1}-\Phi(s))/s$, $s>0$, and $X_0:=A$.
Then for $\zeta\in S_{\pi-\alpha}$ on account of estimate:
%\begin{equation}\label{cond1'}
%
\begin{eqnarray*}
&&\left\| {A\over\zeta \mathds{1} +A+\zeta sA}\cdot{A\over\zeta \mathds{1} +A}\right\| \leq \\
&&\left(1+
{|\zeta|\over \mbox{dist}\left({\zeta (1+s\zeta)^{-1}},-S_\alpha\right)}\right)
\left(1+{|\zeta|\over \mbox{dist} (\zeta,-S_\alpha)}\right) \ ,
\end{eqnarray*}
the family $\{X(s)\}_{s > 0}$ converges, when $s\rightarrow +0$, to $X_0$ in the resolvent
uniform (i.e., \textit{operator-norm}) sense with asymptotic:
\begin{equation*}
\|(\zeta \mathds{1} +X(s))^{-1} - (\zeta \mathds{1} +X_0)^{-1}\| =
s \, \left\| {A\over\zeta \mathds{1} +A+
\zeta sA}\cdot{A\over\zeta \mathds{1} +A}\right\| = O(s) \, .
\end{equation*}
As a consequence of Proposition \ref{prop:6.3.1}, the family $\{\Phi(t)\}_{t\geq 0}$
satisfies the conditions of Proposition \ref{prop:2.1.12}. Then seeing that
for $t>0$ and $n \in \mathbb{N}$ one has estimate
\begin{eqnarray}\label{eq:E1a}
\left\|\Phi(t/n)^n -e^{-tX_0}\right\| && \leq \\
&& \|\Phi(t/n)^n -e^{-t \, X(t/n)}\| +
\|e^{-t \, X(t/n)} - e^{-t \, X_0}\| \, . \nonumber
\end{eqnarray}
%This provides
This provides by (\ref{eq:6.2.5}) for $C = \Phi(t/n)$ and Proposition {\ref{prop:6.3.1} (b)}
the operator-norm approximation
(\ref{eq:2.1.17}), which is the \textit{Euler formula} (\ref{eq:E1}).
\end{proof}
%%%%%%%%%%%%%%%%%%%%%%%%%%%%%%%%%%%%%%%%%%%%%%%%%%%%%%%%%%%%%%%%%%%%%%%%%%%%%%%%%%%%%%%%%%%%%%%%%%%
%\begin{corollary}\label{cor:4.1.2}\emph{(operator-norm Trotter product formula)}

According to (\ref{eq:E1a})
%Proposition {\ref{prop:6.3.1}} and  Proposition {\ref{prop:2.1.12}}
the \textit{rate} of operator-norm convergence of the \textit{Chernoff}
product formula {(\ref{eq:2.1.17})} is determined by convergence rate of the \textit{Chernoff}
estimate {(\ref{eq:6.2.5})} of $\left\|\Phi(t/n)^n - e^{-t \, X(t/n)}\right\|$
%for $C = \Phi(t/n)$
{along with} the rate of convergence in the \textit{Trotter-Neveu-Kato} theorem for
$\left\|e^{-t \, X(t/n)} - e^{-t \, X_0}\right\|$, Proposition {\ref{prop:6.3.1} (b)}.
Then for the \textit{Euler formula} (\ref{eq:E1}) the accuracy of this \textit{two-step} estimate
is limited by the order $O(1/n^{1/3})$ because of the Chernoff estimate (\ref{eq:6.2.5}) on the
the \textit{first} step in the right-hand side of (\ref{eq:E1a}).
%%%%%%%%%%%%%%%%%%%%%%%%%%%%%%%%%%%%%%%%%%%%%%%%%%%%%%%%%%%%%%%%%%%%%%%%%%%%%%%%%%%%%%%%%%%%%%%%%%%
\begin{remark}\label{rem:6.2.2-3}\rm{
Note that the rate $O(1/n^{1/3})$ is far from to be \textit{optimal}, which \textit{a fortiori} is
known as $O(1/n)$. Indeed, in Ref.\citenum{CaZ01} (Theorem 5.1) by a direct \textit{one-step}
telescopic estimate it was shown that the rate of convergence in (\ref{eq:E1}) is at least of the
order $O(\ln(n)/n)$.
}
\end{remark}
%%%%%%%%%%%%%%%%%%%%%%%%%%%%%%%%%%%%%%%%%%%%%%%%%%%%%%%%%%%%%%%%%%%%%%%%%%%%%%%%%%%%%%%%%%%%%%%%%%%
\begin{corollary}\label{cor:4.1.2}\emph{(operator-norm Trotter product formula)}
If contractions $\{\Phi(t) := e^{- t A} e^{- t B}\}_{t \geq 0}$ \emph{(\ref{eq:1.8.64})}
satisfy conditions Proposition \emph{\ref{prop:2.1.12}}, then by necessity part of this assertion
and \emph{(\ref{eq:E1a})} the limit $n\rightarrow \infty$ yields the Trotter product formula
\emph{(\ref{eq:1.8.63})} in the operator-norm topology, cf. Proposition \emph{\ref{prop:1.8.25}}.
The rate of convergence \emph{(}if any\emph{)} is determined by the first and the second steps
in the right-hand side of \emph{(\ref{eq:E1a})}, cf. Theorem \emph{5.3} in
\emph{Ref.\citenum{CaZ01}}.
\end{corollary}
%cf. Proposition \ref{prop:1.8.25}
%%%%%%%%%%%%%%%%%%%%%%%%%%%%%%%%%%%%%%%%%%%%%%%%%%%%%%%%%%%%%%%%%%%%%%%%%%%%%%%%%%%%%%%%%%%%%%%%%%%
\begin{remark}\label{rem:6.2.2-4}\rm{
Note that in contrast to the problem of semigroup approximation (Corollary \ref{cor:4.1.1}),
that needs assumptions only on generator $A$, the approximation by the Trotter product formula
requests a condition on a couple of generators $\{A,B\}$ (Proposition \ref{prop:1.8.25} and
Corollary \ref{cor:4.1.2}).
%As a consequence, in the semigroup approximation theory the one-step estimate is usually sufficient,
%see, e.g., Corollary \ref{cor:4.1.1} for Euler formula (\ref{eq:E1}) and
%Refs.\citenum{GoTo14, GoKoTo19} for more examples.
Since the Trotter product approximants $\{\Phi(t/n)^n\}_{n \geq 1, t \geq 0}$ involve a
\textit{couple} of generators $\{A,B\}$, the proof must take into account a \textit{subordination}
of generators $A$ and $B$.
A variety of the \textit{one-} and \textit{two-step} methods (including the Chernoff estimate),
as well as of the corresponding conditions, that ensure the convergence of the Trotter product
formula, is quite large, see, for example, Chapters 5.1-5.3
in Ref.\citenum{Zag19}. They determine the sense of the algebraic sum $(A + B)$ and the accuracy
of the operator-norm convergence estimates, compare, e.g., {Ref.\citenum{CaZ01}} and
Refs.\citenum{NZ98, ITTZ01}.
%the \textit{two-step} limit (Corollary \ref{cor:4.1.2}) that includes the Chernoff estimate
%and the {Trotter-Neveu-Kato} theorem is \textit{a priori} indispensable.
}
\end{remark}
%%%%%%%%%%%%%%%%%%%%%%%%%%%%%%%%%%%%%%%%%%%%%%%%%%%%%%%%%%%%%%%%%%%%%%%%%%%%%%%%%%%%%%%%%%%%%%%%%%%
%%%%%%%%%%%%%%%%%%%%%%%%%%%%%%%%% Section %%%%%%%%%%%%%%%%%%%%%%%%%%%%%%%%%%%%%%%%%%%%%%%%%%%%%%%%%
\section{Chernoff Estimate and Dunford-Segal Approximation}\label{sec:5}
%%%%%%%%%%%%%%%%%%%%%%%%%%%%%%%%%%%%%%%%%%%%%%%%%%%%%%%%%%%%%%%%%%%%%%%%%%%%%%%%%%%%%%%%%%%%%%%%%%%
Here we continue with more comments about \textit{optimal} results for the rate of
convergence for operator-norm approximants of $C_0$-semigroups mentioned in
Remark \ref{rem:6.2.2-3}. This approximation theory was advanced in Ref.\citenum{GoTo14}
(and improved later in Ref.\citenum{GoKoTo19}) for the \textit{Yosida}, the \textit{Dunford-Segal}
and the \textit{Euler} approximations of $C_0$-semigroups.
In addition to standard analysis of approximations in the
strong operator topology (cf. Ref.\citenum{But20}) the paper Ref.\citenum{GoTo14} proposed to
study a \textit{vector-dependent} estimates of the convergence rate for approximants, see
Remark \ref{rem:1.1.0}.
For holomorphic $C_0$-semigroups this estimates of convergence can be uplifted to the
operator-norm topology.

This last aspect will be considered in the present section versus the Chernoff operator-norm
estimate. We start by citation of the optimal $O(1/n)$ result for the rate of the operator-norm
convergence for a simple case of the Euler approximants (\textit{Euler formula} (\ref{eq:E1})),
which also shows a sectorial dependence of the upper bound.
%As a consequence, in the semigroup approximation theory the one-step estimate is usually sufficient,
%see, e.g., Corollary \ref{cor:4.1.1} for Euler formula (\ref{eq:E1}) and
%Refs.\citenum{GoTo14, GoKoTo19} for more examples.
%%%%%%%%%%%%%%%%%%%%%%%%%%%%% Proposition %%%%%%%%%%%%%%%%%%%%%%%%%%%%%%%%%%%%%%%%%%%%%%%%%%%%%%%%%
\begin{proposition}\label{prop:2.1.12-1}{\emph{(Ref.\citenum{ArZ10})}}
Let $A$ be an $m$-sectorial operator in Hilbert space $\mathfrak{H}$, with semi-angle
$\alpha\in [0,\pi/2)$ and vertex at $z=0$. Then
$\{e^{-t \, A}\}_{\, t \geq 0}$ is a holomorphic quasi-sectorial contraction semigroup and
one infers that
\begin{equation}\label{eq:E2}
\left\|\left(\mathds{1} + {t}A/n \right)^{-n}- e^{-tA}\right\|\le
\frac{M_{\alpha}}{(\cos\alpha)^2 \ n} \ , \ t\ge 0 \ , \  n \in \mathbb{N} \ ,
\end{equation}
where
\begin{equation}\label{eq:2.1.18}
\frac{\pi\sin\alpha}{2\alpha}\le M_{\alpha}\le\min\left(\frac{\pi-\alpha}{\alpha}, \,
M\right)\quad{\rm{and}}\quad\frac{\pi}{2}\le M\le
2+\frac{2}{\sqrt{3}} \ .
\end{equation}
\end{proposition}
%%%%%%%%%%%%%%%%%%%%%%%%%%%%%%%%%%%%%%%%%%%%%%%%%%%%%%%%%%%%%%%%%%%%%%%%%%%%%%%%%%%%%%%%%%%%%%%%%%%

Note that estimate (\ref{eq:2.1.18}) of the coefficient $M_{\alpha}$ in (\ref{eq:E2}) is not
\textit{optimal}. To this aim we elucidate (\ref{eq:2.1.18}) for the case of a non-negative
self-adjoint operator $A\,$, that is, for $\alpha = 0$, cf. Remark \ref{rem:6.2.2-2}.
As a result one obtains by the spectral calculus:
\begin{equation}\label{eq:E2-1}
\left\|(\mathds{1}+tA/n)^{-n} - e^{-tA}\right\| \leq \frac{e^{-1}}{n}\, , \quad t\ge 0 \ ,
\quad  n \in \mathbb{N} \ ,
\end{equation}
cf. one of the error bound in (\ref{eq:6.2.8}).
In Ref.\citenum{GoTo14} (Corollary 1.6 c) the Proposition \ref{prop:2.1.12-1} was extended
to a Banach space and any bounded holomorphic $C_0$-semigroup.

We conclude this section by demonstration that the Chernoff estimate (\ref{eq:6.2.5})
\textit{itself }, i.e., without the two-step construction (\ref{eq:E1a}),
yields for holomorphic $C_0$-semigroups the \textit{Dunford-Segal} approximation
introduced in Ref.\citenum{GoTo14}, Theorem 1.1 b).
%%%%%%%%%%%%%%%%%%%%%%%%%%%%% Theorem %%%%%%%%%%%%%%%%%%%%%%%%%%%%%%%%%%%%%%%%%%%%%%%%%%%%%%%%%
\begin{theorem}\label{th:6.2.3}
If $A$ is an $m$-sectorial operator in  Hilbert space $\mathfrak{H}$, with semi-angle
$\alpha\in [0,\pi/2)$ and vertex at $z=0$, then
\begin{equation}\label{eq:E3}
\left\|e^{-n \,(\mathds{1} - e^{- t\, A/n})} - e^{- t\, A}\right\| \le
\frac{N}{(\cos\alpha)^2 \ n} \ , \quad t\ge 0 \ , \quad  n \in \mathbb{N} \ ,
\end{equation}
for some bounded $N > 0$.
\end{theorem}
%%%%%%%%%%%%%%%%%%%%%%%%%%%%%%%%%%%%%%%%%%%%%%%%%%%%%%%%%%%%%%%%%%%%%%%%%%%%%%%%%%%%%%%%%%%%%%%%%%%
\begin{proof} By virtue of Proposition \ref{prop:2.1.12-1} one gets that
$\{e^{- t\, A/n}\}_{n \geq 1}$ for $t\ge 0$ and each $n \in \mathbb{N}$
is a holomorphic quasi-sectorial contraction $C_0$-semigroup $\{e^{- t\, A/n}\}_{t\ge 0}$.
Let $C : =  e^{- t\, A/n}$. Then by the operator-norm Chernoff estimate
(Theorem \ref{th:6.2.2}) we obtain
\begin{equation}\label{eq:E4}
\left\|e^{- t\, A} - e^{-n \,(\mathds{1} - e^{- t\, A/n})}\right\|
\leq \frac{L_\alpha}{n^{1/3}} \ , \quad t\ge 0 \ , \quad  n \in \mathbb{N} \ .
\end{equation}
This yields the operator-norm approximation of quasi-sectorial contraction semigroup
$\{e^{- t\, A}\}_{t\ge 0}$ by the \textit{Dunford-Segal} approximants
$\{e^{-n \,(\mathds{1} - e^{- t\, A/n})}\}_{n \geq 1}$  with the rate $O(1/n^{1/3})$ on a
Hilbert space $\mathfrak{H}$.

The uplifting the estimate (\ref{eq:E4}) to sectorial dependent optimal rate (\ref{eq:E3})
follows \textit{verbatim} the arguments in the proof of Proposition \ref{prop:2.1.12-1}.
\end{proof}
%{remark}\label{rem:6.2.2-2} Refs.\citenum{GoTo14}

Note that our definition of the \textit{Dunford-Segal} approximants is (slightly) different from
that in Refs.\citenum{GoTo14} (Corollary 1.6 b), where assertion was extended to a Banach space
for any bounded holomorphic $C_0$-semigroup.
%%%%%%%%%%%%%%%%%%%%%%%%%%%%%%%%%%%%%%%%%%%%%%%%%%% Section %%%%%%%%%%%%%%%%%%%%%%%%%%%%%%%%%%%%%%%%
\section{Concluding remarks} \label{sec:6}
%%%%%%%%%%%%%%%%%%%%%%%%%%%%%%%%%%%%%%%%%%%%%%%%%%%%%%%%%%%%%%%%%%%%%%%%%%%%%%%%%%%%%%%%%%%%%%%%%%%%
1. Summarising, we infer that for {quasi-sectorial} contractions (Definition \ref{def:2.1.2})
one obtains, instead of \textit{divergent} (for $n \rightarrow \infty$) Chernoff's estimate
(\ref{eq:2.1.10}), the estimate (\ref{eq:2.1.15}), which converges for
$n \rightarrow \infty$ to zero in the \textit{operator-norm} topology.
Note that the rate $O(1/n^{1/3})$, (\ref{eq:6.2.5}), of this convergence is obtained with
help of the {Poisson representation} and the {Tchebych\"{e}v inequality} in the spirit of the
proof of Lemma \ref{lem:1.8.23}, and that it is not optimal.\\
2. The estimate ${M_{\alpha}}/{n^{1/3}}$ in the $(\sqrt[3]{n})^{-1}$-Theorem \ref{th:6.2.2}
can be improved by a more refined lines of reasoning.
For example, scrutinising our probabilistic arguments in Section \ref{sec:2} one can find a more
precise Tchebych\"{e}v-type bound for estimate of \textit{tails}. This improves the estimate
(\ref{eq:2.1.15}) and provides the rate $O(\sqrt{\ln(n)/n})$, see Ref.\citenum{Pau04} (Theorem 1).
Although again it is possible only for \textit{quasi-sectorial} contractions ensuring due to
Proposition \ref{prop:2.1.10} the \textit{operator-norm} contr\^{o}l (\ref{eq:2.1.15-2}) of
the \textit{central} part.\\
3. The very same improvement permits also to amend the estimate for the rate of convergence in
the \textit{Euler product formula} (\ref{eq:E1}) from $O(\ln(n)/n)$ (Ref.\citenum{CaZ01}
(Theorem 5.1)) to the optimal $O(1/n)$, Ref.\citenum{Pau04} (Theorem 4), cf. the proof of the
estimate $O(1/n)$ in Ref. \citenum{BenPau04} (Theorem 1.3).

A careful analysis of the \textit{numerical range} localisation for quasi-sectorial
contractions \cite{Zag08, ArZ10}, which are generated in a Hilbert space $\mathfrak{H}$ by
$m$-sectorial operators with a semi-angle $\alpha\in [0,\pi/2)$, allows to uplift the
operator-norm estimate
%in Theorem \ref{th:6.2.2} and
%in Corollary \ref{th:6.4.1}
for the rate of convergence of the \textit{Euler formula} (\ref{eq:E2})
to the ultimate optimal $\alpha$-dependent rate $O(1/n)$, see Ref. \citenum{ArZ10} (Theorem 4.1).\\
4. We note that in case of \textit{self-adjoint} contractions $C$ (that is, for $\alpha =0$)
with help of the {spectral representation} one can easily obtain the
optimal rate $O(1/n)$ for the \textit{Ritt property} (\ref{eq:2.1.14}), for the
\textit{Chernoff estimate} (\ref{eq:6.2.8}), as well as for the
\textit{Euler formula} (\ref{eq:E2-1}). \\
5. Theorem \ref{th:6.2.3} is an illustration of a \textit{direct} application of the Chernoff estimate
in the approximation theory of holomorphic $C_0$-semigroups for $m$-sectorial generators in
Hilbert space for a particular case of the \textit{Dunford-Segal} approximants.

Note that in Ref.\citenum{GoTo14} and Ref.\citenum{GoKoTo19} a generalisation on Banach space
%relying on a functional calculi ideas
was developed for approximants involving the \textit{Bernstein functions},
see (1.4) in Ref.\citenum{GoKoTo19}.
We remark that a similar generalisation of approximants is known since
Ref.\citenum{Kato74} as the \textit{Kato functions} for the Trotter-Kato product formul\ae, see
Ref.\citenum{Zag19} (Appendix C) for details.

%We note that with help of the \textit{spectral representation} one can easily obtain in
%the optimal rate $O(1/n)$ of the operator-norm convergence for
%\textit{self-adjoint} contractions $C$. This is a particular case of the quasi-sectorial
%contraction for $\alpha=0$, cf. \cite[Remark 3.2]{CaZ01}.
%This also concerns the optimal rate of convergence $O(1/n)$ for the self-adjoint Euler
%approximation formula (\ref{eq:E1}) for $A = A^* \geq 0$, which is $m$(sectorial operator for
%$\alpha=0$.
%%%%%%%%%%%%%%%%%%%%%%%%%%%%%%%%%%%%%%%%%%%%%%%%%%%%%%%%%%%%%%%%%%%%%%%%%%%%%%%%%%%%%%%%%%%%%%%%%%%%
\frenchspacing
%%%%%%%%%%%%%%%%%%%%%%%%%%%%%%%%% Section %%%%%%%%%%%%%%%%%%%%%%%%%%%%%%%%%%%%%%%%%%%%%%%%%%%%%%%%%%
\section*{Acknowledgements}
The strong convergence with the rate (c) (see, Section \ref{sec:3}) was
established for the first time by T. M\"{o}bus and C. Rouz\'{e} in Ref.\citenum{MR21} (Lemma 4.2)
by the method, which is different from that in our Theorem \ref{th:1.1.1}.
I am thankful to Tim M\"{o}bus for useful correspondences and remarks concerning Ref.
\citenum{Zag17} (Lemma 2.1).
%in particular, for attracting my attention to an inconsistency in \cite[Lemma 2.1]{Zag17},
%which is corrected in the present paper.

These Notes were inspired by my lecture delivered at the International Conference
"Selected Topics in Mathematical  Physics" dedicated to the 75th anniversary of Igor V. Volovich
(27-30 September 2021, Steklov Mathematical Institute, Moscow).
I am grateful to Organizing Committee of the Conference for invitation.
%%%%%%%%%%%%%%%%%%%%%%%%%%%%%%%%% Section %%%%%%%%%%%%%%%%%%%%%%%%%%%%%%%%%%%%%%%%%%%%%%%%%%%%%%%%%%
%\section{References}
%Non BiBTeX users can list down their references as:

\end{document}